\documentclass[12pt]{article}
\usepackage{amsfonts, amsmath}

\newtheorem{theorem}{Theorem}
\newtheorem{corollary}{Corollary}
\newtheorem{lemma}{Lemma}

\newenvironment{proof}
{\noindent{\bf Proof\/}.}{{ $\Box$}\smallskip\par}

\title{On the $S_n$-equivariant Euler characteristic of moduli spaces of hyperelliptic curves.}
\author{E. Gorsky\footnote{Partially supported by the grants RFBR-007-00593, RFBR-08-01-00110-a, NSh-709.2008.1 and the Moebius Contest fellowship for young scientists.}}

\begin{document}
\maketitle

\begin{abstract}
The generating function for $S_n$-equivariant Euler characteristics of moduli spaces of pointed hyperelliptic curves
for any genus $g\ge 2$ is calculated. This answer generalizes the known ones for genera  2 and 3 and answers obtained by J. Bergstrom for any genus and $n\le 7$ points.  
\end{abstract}

\section{Introduction}

Consider the moduli space $\mathcal{H}_{g,n}$ of hyperelliptic curves of genus $g$ with $n$ marked points. One has a natural action of the symmetric group $S_n$ on this space, so its homologies are representations of $S_n$. Let $V_{\lambda}$ be the irreducible representations of $S_n$, $s_{\lambda}$
be the corresponding Schur polynomials and $H^{i}(\mathcal{H}_{g,n})=\sum_{\lambda}a_{i,\lambda}V_{\lambda}$ for some integers $a_{i,\lambda}$. Define the $S_n$-equivariant Euler characteristic of $\mathcal{H}_{g,n}$ by the formula
$$\chi^{S_n}(\mathcal{H}_{g,n})=\sum_{i,\lambda}(-1)^{i}a_{i,\lambda}s_{\lambda}.$$

Let $p_n$ denote the $n$th elementary Newton polynomial in the infinite number of variables. Then  $\chi^{S_n}(\mathcal{H}_{g,n})$  can be also calculated by the formula

$$\chi^{S_n}(\mathcal{H}_{g,n})=\sum_{i}(-1)^{i}\sum_{\sigma\in S_{n}}(-1)^{|\sigma|}p_1^{k_1(\sigma)}\cdot\ldots\cdot p_{n}^{k_n(\sigma)}\cdot \mbox{\rm Tr}(\sigma|_{H^{i}(\mathcal{H}_{g,n})}),$$
where $k_i(\sigma)$ denotes the number of cycles of length $i$ in the permutation $\sigma$.

This text provides an explicit answer for a generating function for $S_n$-equivariant Euler characteristics of moduli spaces $\mathcal{H}_{g,n}$ of hyperelliptic curves of arbitrary genus $g\ge 2$ and $n$ marked points.

This problem was studied by several authors. The answer for $g=0$ is well-known, the answer for $g=1$ was obtained by E. Getzler (\cite{getzler},\cite{getzler1}).  After that it was developed by O. Tommasi and J. Bergstrom in \cite{bergstrom},\cite{bertom},\cite{tommasi}.
For example, O. Tommasi proved that for any genus homologies of the moduli space $\mathcal{H}_{g}$ (without marked points) are trivial, so the Euler characteristic of the corresponding coarse moduli space is equal to 1. Using point counts over finite fields, J. Bergstrom discovered a sequence of recurrence relations between the coefficients of the corresponding characters, which permits him to compute the $S_n$-equivariant Euler characteristics of $\mathcal{H}_{g,n}$ for $n\le 7$ and any $g$.
These answers are quite complicated, for example, for $n=4$ the answer non-trivially depends of the residue of the genus modulo 12. For genus 2 the answer for $n$ small enough was obtained in \cite{getzler2},\cite{getzler3}, and for every $n$ in \cite{my}.
For genus 3 a bunch of answers (at least for $n\le 30$) was obtained by Bini and van den Geer (\cite{bini}). G. Bini also calculated in \cite{bini2} all non-equivariant Euler characteristics of $\mathcal{H}_{g,n}$.

The approach of the current paper extends the one used in (\cite{my}). It is based on consideration of the forgetful map $\mathcal{H}_{g,n}\rightarrow\mathcal{H}_{g}$. Its fiber over a curve $C$ is isomorphic
to $F(C,n)/Aut(C)$, where by $F(X,n)$ from now on we'll denote the configuration space of ordered $n$-tuples of distinct points of a space $X$. E. Getzler in (\cite{getzler}) obtained a formula for the generating function for the $S_n$-equivariant  Hodge-Deligne polynomial of $F(X,n)$ for any $X$. After a slight modification of Getzler's formula the following formula was obtained in (\cite{my}).

Let a finite group $G$ acts on a quasiprojective variety  $X$, and for $g\in G$ we denote by $X_k(g)$ a subset of $X$ consisting of points with $g$-orbits of length $k$. For example, $X_1(g)$ is a set of $g$-fixed points. Then the following equation holds:

\begin{equation}
\label{eqF}
\sum_{n=0}^{\infty}t^{n}\chi^{S_n}(F(X,n)/G)={1\over |G|}\sum_{g\in G}\prod_{k=1}^{\infty}(1+p_{k}t^k)^{\chi(X_k(g))\over k}.
\end{equation}
For example, to get a generating function for non-equivariant Euler characteristic one has to set $p_1=1$ and $p_i=0$ for $i>1$, so
\begin{equation}
\label{neqF}
\sum_{n=0}^{\infty}{t^{n}\over n!}\chi(F(X,n)/G)={1\over |G|}\sum_{g\in G}(1+t)^{\chi(X_1(g))}.
\end{equation}
The last equation can be checked independently, since the generating function for the Euler characteristics of $g$-fixed points on $F(X,n)$ equals to $(1+t)^{\chi(X_1(g))}$.

Now we decompose $\mathcal{H}_g$ on strata $\Xi_G$ consisting of curves with automorphism group $G$. If needed, one should also decompose these strata in such manner that corresponding group actions are isomorphic. Now we obtain a formula for the equivariant:
\begin{equation}
\label{eqH}
\sum_{n=0}^{\infty}t^{n}\chi^{S_n}(\mathcal{H}_{g,n})=\sum_{G}{\chi(\Xi_G)\over |G|}\sum_{g\in G}\prod_{k=1}^{\infty}(1+p_{k}t^k)^{\chi(X_k(g))\over k}
\end{equation}  
and non-equivariant:
\begin{equation}
\label{neqH}
\sum_{n=0}^{\infty}{t^{n}\over n!}\chi(\mathcal{H}_{g,n})=\sum_{G}{\chi(\Xi_G)\over |G|}\sum_{g\in G}(1+t)^{\chi(X_1(g))}
\end{equation}
Euler characteristics of moduli spaces of hyperelliptic curves.

Therefore a priori to obtain an answer one should decompose $\mathcal{H}_g$ on strata corresponding to all possible automorphism group actions, calculate Euler characteristics of the corresponding strata and $X_k(g)$ for all $k$ and $g$. In (\cite{my}) this program was realized for genus 2. For higher genus this program is theoretically doable, because hyperelliptic curves with non-trivial symmetry groups corresponds to symmetric configurations of ramification points on $\mathbb{CP}^1$, but the number of possible  symmetric configurations increases dramatically with genus. Also the structure of these groups becomes very sophisticated, for example, all symmetry groups of regular polyhedra (e.g. of icosahedron) will appear.

In this article we propose a refinement of this approach, namely, we change the order of summation in (\ref{eqH}).
All automorphism groups of different hyperellipic curves contain in a certain extension of the group $PGL(2,\mathbb{C})$. Therefore we can define some natural classes of such automorphisms.
We choose these classes $A$ such that $\chi(X_k(g))$ for all $k$ are constant on $A$ (so we can denote it as $\chi_{k}(A)$), 
and (\ref{eqH}) can be rewritten in a following form:  
 $$\sum_{n=0}^{\infty}t^{n}\chi^{S_n}(\mathcal{H}_{g,n})=\sum_{A}(\sum_{G}{\chi(\Xi_G)\over |G|}\cdot |\{h\in G|h\in A \}|)\cdot \prod_{k=1}^{\infty}(1+p_{k}t^k)^{\chi_k(A)\over k}.$$
The idea is that the sum in parentheses corresponding to every class $A$ has some unexpected nice properties.
For example, for $A=\{e\}$ we get (since every group has unique unit element) the orbifold Euler characteristic of $\mathcal{H}_{g}$. We prove that for all other $A$ the corresponding coefficients are  orbifold Euler characteristics of some configuration spaces. This gives an easy and natural way for computation of these coefficients in our case, what leads to the final answer.

\bigskip

The paper is organized in a following way. In the section 2 we discuss the answer for the non-equivariant Euler characteristic (Theorem 1), which is obtained independently and coincides with the results of \cite{bini2}, in the section 3 we define in a slightly more general setting the coefficients mentioned above and prove some of their properties. Finally, in the section 4 we compute these coefficients for the moduli spaces of hyperelliptic curves and apply them to obtain a final answer for the equivariant case (Theorem 2). In the appendices we check the coincidence of our answer with the one obtained by J. Bergstrom  in (\cite{bergstrom}) for $n\le 4$ marked points
and the formula of G. Bini (\cite{bini2}) for the non-equivariant Euler characteristics.
The coincidence with the results of Bini and van den Geer (\cite{bini}) for genus 3 is also obtained
up to 30 points, but this check is not included in this text. 

\bigskip

The author is grateful to J. Bergstrom for his interest to this work, discussions and providing huge and very informative tables of equivariant Euler characteristics for small number of points. Without the attention of prof. Bergstrom this activity would not be started. The author would like to thank also M. Kazarian and S. Lando for lots of useful discussions.

\bigskip

\section{Non-equivariant answer}

From the discussion in the previous section we get the formula (\ref{neqH}), so the non-equivariant answer has a form 
$$\sum_{n=0}^{\infty}{t^{n}\over n!}\chi(\mathcal{H}_{g,n})=\sum_{k}c_k(1+t)^k,$$
where $c_k$ are some unknown coefficients, and $k$ runs over the set of  Euler characteristics of fixed point sets of all possible automorphisms of a hyperelliptic curve of genus $g$. Such an automorphism can be identical ($k=2-2g$), a hyperelliptic involution ($k=2+2g$), or its restriction onto the underlying $\mathbb{CP}^{1}$ is nontrivial, i.e. has 2 fixed points, and so $k$ can be equal to 0, 1, 2, 3 or 4.
 An important remark is that $c_k=c_{4-k}$, since if an automorphism has $k$ fixed points, then its composition with the involution has $4-k$ fixed points. Another remark is that $c_{2-2g}=c_{2+2g}$ is equal to the orbifold Euler characteristic of $\mathcal{H}_g$, which equals to ${-1\over 2\cdot 2g(2g+1)(2g+2)}$ (e.g \cite{bini2} or \cite{my}).
 
Therefore  we get the following equation:
$$\sum_{n=0}^{\infty}{t^{n}\over n!}\chi(\mathcal{H}_{g,n})={-1\over 2\cdot 2g(2g+1)(2g+2)}[(1+t)^{2-2g}+(1+t)^{2+2g}]+c_0[1+(1+t)^2]+$$
$$c_1[(1+t)+(1+t)^3]+c_2(1+t)^2.$$
The last thing to do is to find unknown $c_0, c_1$ and $c_2$.

Since a theorem of Tommasi (\cite{tommasi}) states that homologies of $\mathcal{H}_{g}$ are trivial, we have $\chi(\mathcal{H}_{g,0})=1$. From the results of Bergstrom (\cite{bergstrom},\cite{bertom}) it follows that $\chi(\mathcal{H}_{g,2})=2$ and $\chi(\mathcal{H}_{g,4})=-2g$. This gives us a system of 3 linear equations for the coefficients, and, after solving it, we get
$$c_0=-{g\over 8(g+1)},\,\, c_1={g\over 2g+1},\,\,\,\, c_2={g+1\over 4g}.$$
Finally, we get the following equation.

\begin{theorem}
\begin{equation}
\label{neqHfin}
\sum_{n=0}^{\infty}{t^{n}\over n!}\chi(\mathcal{H}_{g,n})={-1\over 2\cdot 2g(2g+1)(2g+2)}[(1+t)^{2-2g}+(1+t)^{2+2g}]-{g\over 8(g+1)}[1+(1+t)^2]+\\
\end{equation}
$${g\over 2g+1}[(1+t)+(1+t)^3]+{g+1\over 4g}(1+t)^2.$$
\end{theorem}

\begin{corollary}
If $n>2g+2$, then
$$\chi(\mathcal{H}_{g,n})=(-1)^{n+1}{(2g+n-3)!\over 2\cdot 2g(2g+1)(2g+2)\cdot (2g-3)!}.$$
If $5\le n\le 2g+2$, then
$$\chi(\mathcal{H}_{g,n})=(-1)^{n+1}{(2g+n-3)!\over 2\cdot 2g(2g+1)(2g+2)\cdot (2g-3)!}-{1\over 2}{(2g-1)!\over (2g+2-n)!}.$$
Also
$$\chi(\mathcal{H}_{g,0})=1, \chi(\mathcal{H}_{g,1})=2, \chi(\mathcal{H}_{g,2})=2, \chi(\mathcal{H}_{g,3})=0,$$
$$\chi(\mathcal{H}_{g,4})=-2g, \chi(\mathcal{H}_{g,5})=0.$$
\end{corollary}

This answer was obtained using some external information: known answers
for 0, 2 and 4 points, but it's important to remark that  knowing only these three answers we can reconstruct the whole generating function. The coefficients $c_0,c_1,c_2$ have a nice form, which may be surprising in this approach using solution of a system of linear equations on them. Their properties will be studied in the next section, and in section 4 we'll give another proof of the Theorem 1.

\section{Calculation of the coefficients}

Consider a universal curve, i. e. a universal family $E\rightarrow \mathcal{H}_{g}$.
 The simplest invariant of a family is an orbifold Euler characteristic, which can be written as an integral with respect to the Euler characteristic over a base:
$$\chi^{orb}(E)=\int_{\mathcal{H}_{g}}{1\over |Aut(C)|}d\chi.$$

We suggest a following way for its generalization. All automorphism groups of different hyperellipic curves contain in a certain extension of the group $PGL(2,\mathbb{C})$. Therefore we can define some natural classes of such automorphisms. We can determine if an automorphism of a fiber belongs to a class $A$, and set $N_{A}(C)$ equal to the number of elements of $Aut(C)$ belonging to $A$. We define a following rational number:
\begin{equation}
\label{defA}
\chi^{A}(E)=\int_{\mathcal{H}_{g}}{N_{A}(C)\over |Aut(C)|}d\chi=\sum_{G}{\chi(\Xi_G)\over |G|}N_{A}(G),
\end{equation}
where $\Xi_G$ is a stratum of curves with an automorphism group isomorphic to $G$. 

Consider a space $M$ of pairs $(C,\varphi)$, where $\varphi$ is an automorphism of a curve $C$.
A fiber of the natural projection $$\mu:M\rightarrow \mathcal{H}_{g}$$ over a curve $C$ is exactly $Aut(C)$,
so fibers are discrete and we can induce an orbifold structure from $\mathcal{H}_{g}$ to $M$.

Consider a subspace $M_A$ in this space consisting of pairs with automorphisms from a class $A$. 
\begin{lemma}
The orbifold Euler characteristic of $M_A$ equals to $\chi^{A}(E)$.
\end{lemma}

\begin{proof}
Consider a restriction of the projection $\mu$ on $M_A$. Its fiber over a curve $C$ is a finite set with
$N_{A}(C)$ elements, so we can apply the Fubini formula for the integration with respect to the Euler characteristic:
$$\chi^{orb}(M_A)=\int_{\mathcal{H}_{g}}\chi(\mu^{-1}(C))d\chi^{orb}=\int_{\mathcal{H}_{g}}N_{A}(C){d\chi\over |Aut(C)|}=\chi^{A}(E).$$
\end{proof}

\bigskip

Let us  calculate these coefficients in our case.
Let $B(X,n)$ denote the configuration space of unordered $n$-tuples of points of a given space $X$.

\begin{lemma}
The orbifold Euler characteristic of $B(\mathbb{C}^{*},k)/\mathbb{C}^{*}$ equals to ${(-1)^{1-k}\over k}$.
\end{lemma}
\begin{proof}
We can choose an arbitrary point on $B(\mathbb{C}^{*},k)$ and divide the coordinates of all points by its coordinate. We'll get this point at 1 and $k-1$ distinct points on $\mathbb{C}\setminus\{0,1\}$, and this configuration has no additional symmetries.

Recall that for any space $X$
$$\sum_{m=0}^{\infty}t^{m}\chi(B(X,m))=(1+t)^{\chi(X)}.$$
It follows, for example, from the equation (\ref{neqF}).
Therefore
$$\sum_{m=0}^{\infty}t^{m}\chi(B(\mathbb{C}\setminus\{0,1\},m))=(1+t)^{-1},$$
and $\chi(B(\mathbb{C}\setminus\{0,1\},m))=(-1)^{m}$.

Since the points which goes to 1 can be chosen in $k$ ways, we get
$$\chi^{orb}(B(\mathbb{C}^{*},k)/\mathbb{C}^{*})={\chi(B(\mathbb{C}\setminus\{0,1\},k-1))\over k}={(-1)^{k-1}\over k}.$$
\end{proof}

\begin{lemma}
Let $A_n$ be a class of automorphisms of order $n$ of $\mathbb{C}^{*}$ with $N$ unordered distinct marked points. Then
\begin{equation}
\label{phi}
\chi^{A_n}(B(\mathbb{C}^{*},N)/\mathbb{C}^{*})=(-1)^{1-{N\over n}}{\varphi(n)\over N},
\end{equation} 
where $\varphi(n)$ is the Euler function of $n$, i.e. the number of integers less than $n$ and coprime with $n$.
\end{lemma}

\begin{proof}
As above, define $M_{A_n}(B(\mathbb{C}^{*},N)/\mathbb{C}^{*})$ as the space of pairs
($N$-tuple of points, its automorphism of order $n$). Lemma 1 implies that
$$\chi^{A_n}(B(\mathbb{C}^{*},N)/\mathbb{C}^{*})=\chi^{orb}(M_{A_n}(B(\mathbb{C}^{*},N)/\mathbb{C}^{*})).$$

Consider a map $q:M_{A_n}\rightarrow \mathbb{C}^{*}$, transforming a pair ($N$-tuple of points, its automorphism $h$ of order $n$) to $h$. Image of $q$ consists of $\varphi(n)$ primitive roots of unity.
The fiber $q^{-1}(h)$ consists of all $N$-tuples invariant under $h$.

Let us calculate the orbifold Euler characteristic of $q^{-1}(h)$.
Consider an $N$-tuple of distinct points on $\mathbb{C}^{*}$ invariant under $h$.
At first, let's raise the coordinates of all these points into $n$th power. Now we have ${N\over n}$ unordered distinct points on $\mathbb{C}^{*}$ modulo the action of $\mathbb{C}^{*}$, so for computing the orbifold Euler characteristic we can use the previous lemma -- it gives us $(-1)^{1-{N\over n}}{n\over N}$. Since $n$th power is a $n$-fold covering, we get 
$$\chi^{orb}(q^{-1}(h))=(-1)^{1-{N\over n}}{n\over N}\cdot {1\over n}={(-1)^{1-{N\over n}}\over N}.$$

Now we can apply the Fubini formula:
$$\chi^{A_n}(B(\mathbb{C}^{*},N)/\mathbb{C}^{*})=\chi^{orb}(M_{A_n}(B(\mathbb{C}^{*},N)/\mathbb{C}^{*}))=$$
$$\int_{\mathbb{C}^{*}}\chi^{orb}(q^{-1}(h))d\chi=\varphi(n)\chi^{orb}(q^{-1}(h))=(-1)^{1-{N\over n}}{\varphi(n)\over N}.$$
\end{proof}

\section{The equivariant answer}

The equation (\ref{eqH}) says that
$$\sum_{n=0}^{\infty}t^{n}\chi^{S_n}(\mathcal{H}_{g,n})=\sum_{G}{\chi(\Xi_G)\over |G|}\sum_{g\in G}\prod_{k=1}^{\infty}(1+p_{k}t^k)^{\chi(X_k(g))\over k}.$$

Suppose that we choose a set of classes $\mathcal{A}_j$ of automorphisms of hyperelliptic  genus $g$ curves such that:

1) $\mathcal{A}_{i}\cap \mathcal{A}_{j}=\emptyset$ for $i\neq j$.

2) Every automorphism of a hyperellyptic curve of finite order belongs to one of $\mathcal{A}_j$.

3) For every $k$ and $j$ for $g\in \mathcal{A}_{j}$ the number  $\chi(X_{k}(g))$ does not depend on $g$.
Therefore in can be denoted as $\chi_{k}(\mathcal{A}_j)$.

From the definition of $\chi^{\mathcal{A}_j}(\mathcal{H}_{g})$ it is clear that under these conditions the equation (\ref{eqH}) can be rewritten in a form
$$\sum_{n=0}^{\infty}t^{n}\chi^{S_n}(\mathcal{H}_{g,n})=\sum_{j}\chi^{\mathcal{A}_j}(\mathcal{H}_{g})\prod_{k=1}^{\infty}(1+p_{k}t^k)^{\chi_k(\mathcal{A}_{j})\over k}.$$

Now let us describe these classes.
Recall that if an automorphism of a hyperelliptic curve is not equal to the identity of to the hyperelliptic involution, its restriction on $\mathbb{CP}^1$ has 2 fixed points. 
We distinguish all possible symmetries of a hyperelliptic curve by the structure of fixed points and the order $n$ of their restriction on $\mathbb{CP}^{1}$. For each stratum we calculate the corresponding coefficient $\chi^{\mathcal{A}_j}(\mathcal{H}_{g})$ and
the structure of all orbits on a covering. By $N$ we denote the number of non-fixed ramification points on $\mathbb{CP}^{1}$.

\bigskip

1. Identity. The coefficient equals to $-{1\over 2\cdot 2g(2g+1)(2g+2)}$ as the orbifold Euler characteristic of $\mathcal{H}_{g}$, the monomial equals to $(1+p_1t)^{2-2g}$.

\medskip

2. Hyperelliptic involution. The coefficient is the same, all non-ramification points have order 2, so the monomial equals to $(1+p_1t)^{2+2g}(1+p_2t^2)^{-2g}$.

\medskip

3. One ramification point is fixed and fibers over the second fixed point do not interchange.
In this case $N=2g+1$, and by lemma 3 the coefficient equals to $(-1)^{1-{n\over N}}{\varphi(n)\over 2N}$
(the factor ${1\over 2}$ is added because of a hyperelliptic involution), and since $N$ is odd, the coefficient equals ${\varphi(n)\over 2(2g+1)}$. All points except three fixed ones have order $n$,
so the monomial equals to $(1+p_1t)^{3}(1+p_nt^n)^{-{2g+1\over n}}$.

\medskip

4. One ramification point is fixed and fibers over the second fixed point  interchange.
The factor is exactly the same as in the previous case, but the structure of orbits is slightly different: one point is fixed, preimages of the second fixed point gives an orbit of length 2, all other ramification points have order $n$, and all other points have order $2n$. Hence the monomial equals to
$(1+p_1t)^{1}(1+p_2t^2)(1+p_nt^n)^{2g+1\over n}(1+p_{2n}t^{2n})^{-{2g+1\over n}}$.

\medskip

5. No ramification points are fixed ($N=2g+2$), ${N\over n}$ is even and $n$ is even.
In this case the preimages of fixed points simultaneously interchange or do not interchange.
The coefficient is $-{\varphi(n)\over 4N}$: from lemma 3 we get $-{\varphi(n)\over N}$, but we should multiply it by one ${1\over 2}$  because of the involution, and by the second one  because we cannot distinguish the fixed points. If the preimages of the fixed points are fixed, we have 4 fixed points and all other points have order $n$, and the monomial equals to  
$(1+p_1)^4(1+p_nt^n)^{-{2g+2\over n}}$. If they are not fixed, we have 4 points of order 2 and all other points have order $n$, so the monomial equals to $(1+p_2t^2)^2(1+p_nt^n)^{-{2g+2\over n}}$.

\medskip

6. No ramification points are fixed ($N=2g+2$), ${N\over n}$ is even and $n$ is odd.
The coefficient is the same as in the previous case, but the structure of orbits is different if fibers interchange: ramification points have order $n$, but generic points have order $2n$, so the corresponding monomial equals to $(1+p_2t^2)^2(1+p_nt^n)^{{2g+2\over n}}(1+p_{2n}t^{2n})^{-{2g+2\over n}}$.

\medskip

7. No ramification points are fixed ($N=2g+2$), ${N\over n}$ is odd, so $n$ is even.
In this case the fibers over one of fixed points interchange, and over second one do not interchange,
so we can distinguish fixed points and the coefficient equals to ${\varphi(n)\over 2N}$.
We have 2 fixed points, 2 points of the order 2, and all other points have order $n$, so the corresponding monomial has a form $(1+p_1t)^2(1+p_2t^2)(1+p_nt^n)^{-{2g+2\over n}}$.

\medskip

8. Two of ramification points are fixed ($N=2g$), $n$ is odd. 
In this case the coefficient equals $-{\varphi(n)\over 4N}$, since we cannot distinguish fixed points
and $N/n$ is even.
There are two possible structures of orbits: $g^n$ can be identical or an involution.
In the first case we have 2 fixed points, all other have order $n$ and the monomial equals to
$(1+p_1t)^2(1+p_nt^n)^{-{2g\over n}}$. In the second only ramification points have order $n$, and
the monomial equals to $(1+p_1t)^2(1+p_nt^n)^{2g\over n}(1+p_{2n}t^{2n})^{-{2g\over n}}$.

\medskip

9. Two of ramification points are fixed ($N=2g$), $n$ is even. 
The coefficient equals $(-1)^{1-{N\over n}}{\varphi(n)\over 2N}$, $n$th power of an automorphism is an involution, so the monomial equals to $(1+p_1t)^2(1+p_nt^n)^{2g\over n}(1+p_{2n}t^{2n})^{-{2g\over n}}$.

\bigskip

We get a final answer:

\begin{theorem}
$$\sum_{k=0}^{\infty}t^{k}\chi^{S_k}(\mathcal{H}_{g,k})=$$
$$-{1\over 2\cdot 2g\cdot (2g+1)\cdot  (2g+2)}[(1+p_1t)^{2-2g}+(1+p_1t)^{2+2g}(1+p_2t^2)^{-2g}]+$$
$$\sum_{n|(2g+1)}{\varphi(n)\over 2(2g+1)}[(1+p_1t)^{3}(1+p_nt^n)^{-{2g+1\over n}}+(1+p_1t)^{1}(1+p_2t^2)(1+p_nt^n)^{2g+1\over n}(1+p_{2n}t^{2n})^{-{2g+1\over n}}]-$$
$$\sum_{n|(g+1),2|n}{\varphi(n)\over 4(2g+2)}[(1+p_1t)^4(1+p_nt^n)^{-{2g+2\over n}}+(1+p_2t^2)^2(1+p_nt^n)^{-{2g+2\over n}}]-$$
$$\sum_{n|(g+1),2\not|n}{\varphi(n)\over 4(2g+2)}[(1+p_1t)^4(1+p_nt^n)^{-{2g+2\over n}}+(1+p_2t^2)^2(1+p_nt^n)^{{2g+2\over n}}(1+p_{2n}t^{2n})^{-{2g+2\over n}}]+$$
$$\sum_{n|2g+2,n\not| g+1}{\varphi(n)\over 2(2g+2)}(1+p_1t)^2(1+p_2t^2)(1+p_nt^n)^{-{2g+2\over n}}-$$
$$\sum_{n|g,2\not|n}{\varphi(n)\over 4\cdot 2g}[(1+p_1t)^2(1+p_nt^n)^{-{2g\over n}}+(1+p_1t)^2(1+p_nt^n)^{2g\over n}(1+p_{2n}t^{2n})^{-{2g\over n}}]-$$
$$\sum_{n|2g,2|n}(-1)^{1-{2g\over n}}{\varphi(n)\over 2\cdot 2g}(1+p_1t)^2(1+p_nt^n)^{2g\over n}(1+p_{2n}t^{2n})^{-{2g\over n}}.$$

Everywhere we assume $n>1$.
\end{theorem}

\bigskip

It is useful to recall two identities with the Euler function:
$$\sum_{a|n}\varphi(a)=n$$
and $$\sum_{a|n}(-1)^{n/a}\varphi(a)=0,\,\,\,\,\mbox{\rm if $n$ is even}.$$

\bigskip

Let us check the correlation of this answer with the one for non-equivariant case (obtained in the Theorem 1).
If we set all $p_i=0$ for $i>1$ and $p_1=1$, then we'll get
$$\sum_{k=0}^{\infty}{t^{k}\over k!}\chi(\mathcal{H}_{g,k})=-{1\over 2\cdot 2g\cdot (2g+1)\cdot  (2g+2)}[(1+t)^{2-2g}+(1+t)^{2+2g}]+$$
$$\sum_{n|2g+1}{\varphi(n)\over 2(2g+1)}[(1+t)^3+(1+t)]-\sum_{n|g+1}{\varphi(n)\over 4(2g+2)}[(1+t)^4+1]+$$
$$\sum_{n|2g}(-1)^{1-{2g\over n}}{\varphi(n)\over 2\cdot 2g}(1+t)^2+\sum_{n|2g+2,n\not|
g+1}{\varphi(n)\over 2(2g+2)}(1+t)^2.$$

We have $$\sum_{n|2g+1,n>1}\varphi(n)=(2g+1)-1=2g, \sum_{n|g+1,n>1}\varphi(n)=(g+1)-1=g,$$
$$\sum_{n|2g,n>1}(-1)^{1-{2g\over n}}\varphi(n)=0-(-1)=1, \sum_{n|2g+2,n\not{|}g+1}\varphi(n)=(2g+2)-(g+1)=g+1,$$
so
$$\sum_{k=0}^{\infty}{t^{k}\over k!}\chi(\mathcal{H}_{g,k})=-{1\over 2\cdot 2g\cdot (2g+1)\cdot  (2g+2)}[(1+t)^{2-2g}+(1+t)^{2+2g}]+$$
$${g\over 2g+1}[(1+t)^3+(1+t)]-{g\over 8(g+1)}[(1+t)^4+1]+[{1\over 4g}+{1\over 4}](1+t)^{2},$$
that is a correct answer for the non-equivariant Euler characteristics.

\medskip

For $g=2$ we get the same answer as in \cite{my}.

\medskip

\bigskip

\centerline{\large \bf Appendix A: comparison up to 4 points}

\bigskip

It is important to check this answer for small number of points, say, modulo $t^5$.
From the first line of the Theorem 2 we always get 

$${-1\over 2\cdot 2g\cdot (2g+1)\cdot (2g+2)}[2 +4p_1\cdot t + ((2+4g^2)p_1^2 - 
 2g\cdot p_2)t^2 + (4g^2\cdot p_1^3 - 4g(g+1)\cdot p_1 p_2) t^3 + $$
$$({4g^4-g^2\over 3}p_1^4 +(-4g^3-6g^2-2g) p_1^2 p_2 +(2g^2+g)p_2^2) t^4].$$
\bigskip
In what follows it is convenient to introduce $M_k(n)$ as 1 if $g=n\,\,\, (\mod k)$ and 0 otherwise.
\bigskip

From the second line we get
$${g\over 2g+1}[2+4p_1t+(3p_1^2+p_2)t^2+(p_1^3+p_1p_2)t^3]-M_{3}(1){2\over 3}p_1p_3t^4,$$ 

from  the next two lines we get

$${-g\over 8(g+1)}[2+4p_1t+(6p_1^2+2p_2)t^2+4p_1^3t^3+(p_1^4+p_2^2)t^4]-{M_2(1)\over 8}[-2p_2t^2-4p_1p_2t^3+(gp_2^2-6p_1^2p_2)t^4]+$$
$${2M_3(2)\over 3}p_1p_3t^4+{M_4(3)\over 4}p_4t^4,$$

from the next one we get

$${1\over 4}[1+2p_1t+(p_1^2+p_2)t^2+2p_1p_2t^3+p_1^2p_2t^4]+{M_2(0)\over 4}[-p_2t^2-2p_1p_2t^3+(-p_1^2p_2+{g\over 2}p_2^2)t^4]-{M_4(1)\over 4}p_4t^4,$$

and from the last two we get 

$${1\over 4g}[1+2p_1t+p_1^2t^2]+{(-1)^{1-g}\over 4}[p_2t^2+2p_1p_2t^3+(p_1^2p_2+{g-1\over 2}p_2^2-p_4)t^4]+{M_2(0)(-1)^{1-{g\over 2}}\over 4}p_4t^4.$$

As we checked before, all coefficients at powers of $p_1$ (which correspond to the non-equivariant Euler characteristic) are correct. Coefficient at $p_3$ always vanishes, coefficient at $p_1p_3$ is ${2\over 3}(M_3(2)-M_3(1))$, so it is also correct.

Coefficient at $p_2$ equals to $${1\over 2(2g+1)(2g+2)}+{g\over 2g+1}-{g\over 4(g+1)}+{M_2(1)\over 4}+{1\over 4}-{M_2(0)\over 4}+{(-1)^{1-g}\over 4}={1-M_2(0)+M_2(1)\over 2},$$
that is 0 for even $g$ and 1 for odd.

Coefficient at $p_1p_2$ equals to $${1\over 2(2g+1)}+{g\over 2g+1}+{M_2(1)\over 2}+1/2-M_2(0)+{M_2(1)-M_2(0)\over 2}=1-M_2(0)+M_2(1),$$
that is 0 for even $g$ and 2 for odd.

Coefficient at $p_1^2p_2$ equals to $${2g^2+3g+1\over 2(2g+1)(g+1)}+{3M_2(1)\over 4}+{1\over 4}-{M_2(0)\over 4}+{(-1)^{1-g}\over 4}={1\over 2}-{M_2(0)\over 2}+M_2(1),$$
that is 0 for even $g$ and ${3\over 2}$ for odd ones.

Coefficient at $p_2^2$ equals to $$-{1\over 8(g+1)}-{g\over 8(g+1)}-{g\over 8}M_2(1)+{g\over 8}M_2(0)+{(g-1)(-1)^{1-g}\over 8}=-{M_2(1)\over 4},$$
what is true.

Coefficient at $p_4$ equals to $${M_4(3)\over 4}-{M_4(1)\over 4}+{(-1)^g\over 4}+{M_2(0)(-1)^{1-{g\over 2}}\over 4},$$
that is 0 for $g=0 (\mod 4)$, is $-{1\over 2}$ for $g=1 (\mod 4)$, is ${1\over 2}$ for $g=2 (\mod 4)$ and is again 0 for
$g=3 (\mod 4)$.

So we finally can conclude that up to 4 points for any genus the conjectural answer coincides with known before in (\cite{bergstrom}).

\bigskip

\centerline{\large \bf Appendix B: comparison with the results of G. Bini}

\bigskip

G. Bini proved in \cite{bini2} that for $5\le n\le 2g+2$ the following identity holds:

$$\chi(\mathcal{H}_{g,n})=-{(-2)^n\cdot n!\over 2(2g+2)!}((2g-1)!{2g-1+n\choose n}-{(2g)!\over 4}{2g+n-2\choose n-2}+{(2g+1)!\over 32}{2g+n-3\choose n-4}$$
$$+\sum_{r=3}^{[n/2]}{(-1)^r(2g-1)!\over 2^{2r}}{2g-1+r\choose r}{2g-1+n-r\choose n-2r})$$
$$+{(-2)^{n}n!\over 4(2g+1)!}((2g-1)!{2g+n-2\choose n-1}-{(2g)!\over 4}{2g+n-3\choose n-3}$$
$$+\sum_{r=2}^{[(n-1)/2]}{(-1)^r(2g-1)!\over 2^{2r}}{2g-1+r\choose r}{2g-2+n-r\choose n-1-r})$$
$$-{(-2)^{n}n!\over 16(2g)!}(2g-1)!{2g-3+n\choose n-2}-(2g-1)(2g-2)\ldots(2g-n+3)$$
$$-{(-2)^{n}n!\over 16(2g)!}(\sum_{r=1}^{[(n-2)/2]}{(-1)^{r}(2g-1)!\over 2^{2r}}{2g-1+r\choose r}{2g-3+n-r\choose n-2-2r})$$
$$-{(-2)^{n}n!\over 2}\sum_{j=3}^{n-1}{(-1)^{j}(j-3)!\over 2^j j!}\sum_{r=0}^{[(n-j)/2]}{(-1)^r\over 2^{2r}}{j+r-3\choose r}{2g-1+r\choose 2g+2-j}{2g-1+n-j-r\choose n-j-2r}.$$

To check that this answer coincides with the expected one, we first make this formula more compact.

Remark that:
$$(2g-1)!{2g-1+n\choose n}={(2g+n-1)!\over n!},$$
$$(2g)!{2g+n-2\choose n-2}={(2g+n-2)!\over (n-2)!},$$
$$(2g+1)!{2g+n-3\choose n-4}={(2g+n-3)!\over (n-4)!},$$
$$(2g-1)!{2g-1+r\choose r}{2g-1+n-r\choose n-2r}=
{(2g-1+n-r)!\over r!(n-2r)!}.$$
Remark that three first summands in the first bracket have the same form for $r=0,1$ and 2 respectively.
Thus the first bracket can be written in a form
\begin{equation}
-{(-2)^n\cdot n!\over 2(2g+2)!}\sum_{r=0}^{[n/2]}{(-1)^r\over 2^{2r}}{(2g-1+n-r)!\over r!(n-2r)!}.
\end{equation}
Analogously the second sum can be rewritten in a form
\begin{equation}
{(-2)^{n}n!\over 4(2g+1)!}\sum_{r=0}^{[(n-1)/2]}{(-1)^r\over 2^{2r}}{(2g-2+n-3)!\over r!(n-1-2r)!}.
\end{equation}

Further,
$$(2g-1)!{2g-3+n\choose n-2}={2g-3+n\over (n-2)!},$$
$$(2g-1)!{2g-1+r\choose r}{2g-3+n-r\choose n-2-2r}={(2g-3+n-r)!\over r!(n-2-2r)!},$$
and, finally,
$${(j-3)!\over j!}{j+r-3\choose r}{2g-1+r\choose 2g+2-j}{2g-1+n-j-r\choose n-j-2r}={(2g-1+n-j-r)!\over j!r!(2g+2-j)!(n-j-2r)!}.$$

Remark that above we have summands of the same form for $j=0,1$ and 2 respectively.
Therefore Bini's theorem can be reformulated in the following way:
$$\chi(\mathcal{H}_{g,n})=-{(-2)^n n!\over 2}\sum_{j=0}^{n-1}\sum_{r=0}^{[(n-j)/2]}{(-1)^{j+r}\over 2^{j+2r}}{(2g-1+n-j-r)!\over j!r!(2g+2-j)!(n-j-2r)!}-$$
$$-(2g-1)(2g-2)\ldots(2g-n+3).$$

Let us calculate also a summand corresponding to $j=n$, $r=0$. We get

$$-{n!\over 2}{(2g-1)!\over n!(2g+2-n)!}=-{1\over 2}(2g-1)(2g-2)\ldots(2g-n+3),$$
so we have
\begin{equation}
\label{chi}
\chi(\mathcal{H}_{g,n})=-{(-2)^n n!\over 2}\sum_{0
\le j+2r\le n}{(-1)^{j+r}\over 2^{j+2r}}{(2g-1+n-j-r)!\over j!r!(2g+2-j)!(n-j-2r)!}-
\end{equation}
$$-{1\over 2}(2g-1)(2g-2)\ldots(2g-n+3).$$

\begin{lemma}
$$-{(-2)^n n!\over 2}\sum_{0
\le j+2r\le n}{(-1)^{j+r}\over 2^{j+2r}}{(2g-1+n-j-r)!\over j!r!(2g+2-j)!(n-j-2r)!}=$$
$$(-1)^{n+1}{(2g+n-3)!\over 2\cdot 2g(2g+1)(2g+2)(2g-3)!}.$$
\end{lemma}

\begin{corollary}
The Bini's answer coincides with the one obtained in Theorem 1.
\end{corollary}

The corollary is clear since the term outside the sum corresponds to the generating function
$$-{1\over 2\cdot 2g(2g+1)(2g+2)}(1+t)^{2+2g}.$$

Now let us prove the lemma. Its statement can be reformulated as
$$\sum_{0\le j+2r\le n}(-1)^{n-j-r}2^{n-j-2r}{(2g-1+n-j-r)!\over j!r!(2g+2-j)!(n-j-2r)!}=(-1)^{n}{(2g+n-3)!\over 2g(2g+1)(2g+2)n!(2g-3)!}.$$
Let us compare the generating functions of the expressions by $n$ in both parts.
Let $a=n-j-r, b=n-j-2r$. Then $r=a-b$, $n=2a+j-b$.
We have
$$\sum_{n=0}^{\infty}t^{n}\sum_{0\le j+2r\le n}(-1)^{n-j-r}2^{n-j-2r}{(2g-1+n-j-r)!\over j!r!(2g+2-j)!(n-j-2r)!}=$$
$$=\sum_{j=0}^{\infty}\sum_{a=0}^{\infty}\sum_{b=0}^{a}t^{2a+j-b}(-1)^{a}2^{b}{(2g-1+a)!\over j!(a-b)!(2g+2-j)!b!}=\sum_{j=0}^{2g+2}{1\over (2g+2)!}t^{j}{2g+2\choose j}\times$$
$$\sum_{a=0}^{\infty}t^{2a}(-1)^{a}(2g-1+a)!\sum_{b=0}^{a}{1\over a!}t^{-b}2^{b}{a\choose b}=$$
$${(1+t)^{2g+2}\over (2g+2)!}\sum_{a=0}^{\infty}t^{2a}(-1)^{a}(1+{2\over t})^{a}{(2g-1+a)!\over a!}=$$
$${(1+t)^{2g+2}\over (2g+2)!}\sum_{a=0}^{\infty}(-2t-t^2)^{a}{(2g-1+a)!\over a!}=$$
$${(1+t)^{2g+2}\over (2g+2)!}(2g-1)!(1-(-2t-t^2))^{-2g}={(1+t)^{2g+2}\over 2g(2g+1)(2g+2)}(1+t)^{-4g}=$$
$${(1+t)^{2-2g}\over 2g(2g+1)(2g+2)}=\sum_{n=0}^{\infty}t^{n}(-1)^{n}{(2g+n-3)!\over 2g(2g+1)(2g+2)n!(2g-3)!}.$$

This completes the proof. Everywhere in this calculation factorials of negative integers are supposed to be zero.


Moscow State University,\newline
Department of Mathematics and Mechanics.\newline
E. mail: gorsky@mccme.ru.

\end{document}